\documentclass[11pt]{article}
\usepackage{amsmath, amssymb, amsthm, amsfonts}
\usepackage{graphicx}
\usepackage{verbatim}
\usepackage{float}
\usepackage{color}
\usepackage{tikz}
\usepackage{hyperref}

\newtheorem{theorem}{Theorem} 
\newtheorem{lemma}[theorem]{Lemma} 
\newtheorem{proposition}[theorem]{Proposition}

\newtheorem{conjecture}[theorem]{Conjecture}

\theoremstyle{definition}

\theoremstyle{remark}

\begin{document}

\title{Minimal 3-regular Penny Graph}
\author{Alexander Karabegov \and Tanya Khovanova}
\date{}

\maketitle

\begin{abstract}
We prove that a 3-regular penny graph has at least 16 vertices and show that such a graph with 16 vertices exists.
\end{abstract}

\section{Introduction}

The first author met the second author during the All-Soviet Math Olympiad in Yerevan in 1974. Alexander was one year older than Tanya. In 1976, Tanya was still competing as a high school senior, while Alexander was already a freshman at Moscow State University. Alexander proposed the following two related puzzles for the 1976 Moscow Olympiad \cite{GT}.

\begin{quote}
    \textbf{Puzzle 1.} You are given a finite number of points in a plane. Prove that there exists a point with at most 3 closest neighbors.
\end{quote}

Just in case, by the \textit{closest neighbors} we mean all points at the minimal distance from a given point. This puzzle is connected to our main discussion, so we provide its solution in Section~\ref{sec:prelim}. We encourage the reader to solve it before reading our solution.

This puzzle implies that if a planar configuration is such that every point has the same number of closest neighbors, then that number is no greater than 3. This is where the second problem comes in.

\begin{quote}
    \textbf{Puzzle 2.} Can you place a finite number of points on the plane in such a way that each point has exactly $3$ closest neighbors?
\end{quote}

The second problem has an elegant solution with $24$ points chosen from a triangular grid, as seen in Figure~\ref{fig:24pointssolution}.

\begin{figure}[ht!]
\centering
    \begin{tikzpicture}[scale=0.8, line join=round]
  \coordinate (A1) at ( -2, 0);
  \coordinate (B1) at ( -2.5, {sqrt(3)/2});
  \coordinate (C1) at (-3, 0);
  \coordinate (D1) at ( -2.5, {-sqrt(3)/2});
   \coordinate (D2) at (-2, {sqrt(3)});
   \coordinate (A2) at ( -1, {sqrt(3)});
   \coordinate (C2) at ( -1.5, {3*sqrt(3)/2});
   \coordinate (B2) at ( -.5, {3*sqrt(3)/2});
   \coordinate (D3) at ( .5, {3*sqrt(3)/2});
   \coordinate (C3) at ( 1.5, {3*sqrt(3)/2});
   \coordinate (A3) at ( 1, {sqrt(3)});
   \coordinate (B3) at ( 2, {sqrt(3)});
   \coordinate (D4) at (-2, {-sqrt(3)});
   \coordinate (A4) at ( -1, {-sqrt(3)});
   \coordinate (C4) at ( -1.5, {-3*sqrt(3)/2});
   \coordinate (B4) at ( -.5, {-3*sqrt(3)/2});
   \coordinate (D5) at ( .5, {-3*sqrt(3)/2});
   \coordinate (C5) at ( 1.5, {-3*sqrt(3)/2});
   \coordinate (A5) at ( 1, {-sqrt(3)});
   \coordinate (B5) at ( 2, {-sqrt(3)});
   \coordinate (B6) at ( 2.5, {sqrt(3)/2});
   \coordinate (A6) at ( 2, 0);
   \coordinate (C6) at ( 3, 0);
   \coordinate (D6) at ( 2.5, {-sqrt(3)/2});
  \draw[thick] (A1)--(B1)--(C1)--(D1)--cycle;
  \draw[thick] (A1)--(C1);
  \draw[thick] (B1)--(D2);
  \draw[thick] (A2)--(D2);
  \draw[thick] (A2)--(C2);
  \draw[thick] (C2)--(D2);
  \draw[thick] (A2)--(B2);
  \draw[thick] (B2)--(C2);
  \draw[thick] (B2)--(D3);
  \draw[thick] (D3)--(C3);
  \draw[thick] (D3)--(A3);
  \draw[thick] (B3)--(C3);
  \draw[thick] (A3)--(B3);
  \draw[thick] (A3)--(C3);
  \draw[thick] (B6)--(B3);
  \draw[thick] (B6)--(C6);
  \draw[thick] (B6)--(A6);
  \draw[thick] (A6)--(C6);
  \draw[thick] (D6)--(C6);
  \draw[thick] (D6)--(A6);
  \draw[thick] (D1)--(D4);
  \draw[thick] (D4)--(A4);
  \draw[thick] (D4)--(C4);
  \draw[thick] (A4)--(C4);
  \draw[thick] (B4)--(A4);
  \draw[thick] (B4)--(C4);
  \draw[thick] (B4)--(D5);
  \draw[thick] (B5)--(C5);
  \draw[thick] (B5)--(A5);
  \draw[thick] (A5)--(C5);
  \draw[thick] (A5)--(D5);
  \draw[thick] (C5)--(D5);
  \draw[thick] (B5)--(C6);
  \foreach \P in {A1,B1,C1,D1,D2,A2,B2, C2,C3, D3,A3,B3,A4,B4,C4,D4,A5,B5,C5,D5,B6,A6,C6,D6}
    \fill (\P) circle[radius=0.5mm];
\end{tikzpicture}
\caption{A solution with 24 points}
\label{fig:24pointssolution}
\end{figure}
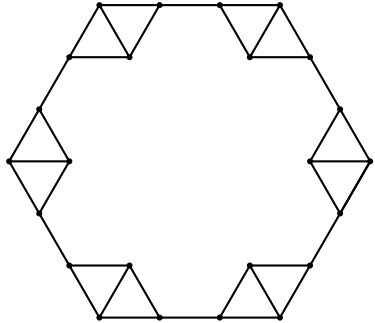

By the way, Tanya competed in the 1976 Moscow Olympiad and was given the first problem to solve; she is sure she solved it. The story continued almost $40$ years later, when Alexander emailed Tanya a configuration with $16$ points in Figure~\ref{fig:minimal} and the following conjecture.

\begin{figure}[ht!]
\centering
    \begin{tikzpicture}[scale=1., line join=round]
    \begin{scope}[rotate=45]
    \begin{scope}[shift={({sqrt(3)/2 + sqrt(2)/2},0)}, rotate=90]
  \coordinate (A1) at ( 1/2, 0);
  \coordinate (B1) at ( 0, {sqrt(3)/2});
  \coordinate (C1) at (-1/2, 0);
  \coordinate (D1) at ( 0, {-sqrt(3)/2});
 \end{scope}
  \draw[thick] (A1)--(B1)--(C1)--(D1)--cycle;
  \draw[thick] (A1)--(C1);
  \foreach \P in {A1,B1,C1,D1}
    \fill (\P) circle[radius=0.5mm];
    \begin{scope}[shift={({sqrt(3)/2 + sqrt(2)/2},{-(sqrt(3)+sqrt(2))})}, rotate=90]
  \coordinate (A2) at ( 1/2, 0);
  \coordinate (B2) at ( 0, {sqrt(3)/2});
  \coordinate (C2) at (-1/2, 0);
  \coordinate (D2) at ( 0, {-sqrt(3)/2});
 \end{scope}
  \draw[thick] (A2)--(B2)--(C2)--(D2)--cycle;
  \draw[thick] (A2)--(C2);
  \foreach \P in {A2,B2,C2,D2}
    \fill (\P) circle[radius=0.5mm];
   \begin{scope}[shift={(0,{-(sqrt(3)+sqrt(2))/2})}]
  \coordinate (A3) at ( 1/2, 0);
  \coordinate (B3) at ( 0, {sqrt(3)/2});
  \coordinate (C3) at (-1/2, 0);
  \coordinate (D3) at ( 0, {-sqrt(3)/2});
 \end{scope}
  \draw[thick] (A3)--(B3)--(C3)--(D3)--cycle;
  \draw[thick] (A3)--(C3);
  \foreach \P in {A3,B3,C3,D3}
    \fill (\P) circle[radius=0.5mm];
      \begin{scope}[shift={({sqrt(2)+sqrt(3)},{-(sqrt(3)+sqrt(2))/2})}]
  \coordinate (A4) at ( 1/2, 0);
  \coordinate (B4) at ( 0, {sqrt(3)/2});
  \coordinate (C4) at (-1/2, 0);
  \coordinate (D4) at ( 0, {-sqrt(3)/2});
 \end{scope}
  \draw[thick] (A4)--(B4)--(C4)--(D4)--cycle;
  \draw[thick] (A4)--(C4);
  \foreach \P in {A4,B4,C4,D4}
    \fill (\P) circle[radius=0.5mm];
    \draw[thick] (D2)--(D4);
    \draw[thick] (B1)--(B3);
     \draw[thick] (D1)--(B4);
     \draw[thick] (B2)--(D3);
     \end{scope}
\end{tikzpicture}
\caption{A solution with 16 points}
\label{fig:minimal}
\end{figure}
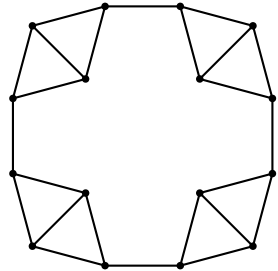

\begin{conjecture}[Karabegov’s Conjecture]
     Any finite planar point configuration in which every point has exactly $3$ closest neighbors must contain at least $16$ points.
\end{conjecture}

In this paper, we prove the conjecture. We start with preliminary observations and definitions.

\section{Preliminaries}
\label{sec:prelim}

We consider a finite number of points on a plane.

As we promised, we provide a solution to the first puzzle.

We prove it by contradiction. Suppose we built a set $M$ of points on a plane, so that each point has at least $4$ closest neighbors. Suppose $d$ is the smallest among all these distances. Suppose $M' \subset M$ is the subset of all points such that the distance from these points to their closest neighbors is $d$. Obviously, the set $M'$ is non-empty, and each point in it has at least $4$ closest neighbors, all of which must belong to $M'$.

Now, we build the convex hull of $M'$, which is a convex polygon.
We denote it by $H$. Suppose $A$ is one of the extreme points in $H$. Recall that an \textit{extreme point} on a convex hull is a point that is not a convex linear combination of any other points of the hull. The point $A$ is a vertex of $H$ and the internal angle at $A$ is less than $180^\circ$.

Let $B_1$, $B_2$, $B_3$, and $B_4$ be $4$ points in $M'$ that are a distance $d$ from $A$. Consider two distinct neighbors $B_i$ and $B_j$ of $A$. Any angle $\angle B_iAB_j$ has to be at least $60^\circ$, as $|B_iB_j| \geq d$. This contradicts our conclusion that $B_1$, $B_2$, $B_3$, and $B_4$ all fit inside an interior angle of less than $180^\circ$.

We are now back to building sets in which each point has exactly 3 closest neighbors. We consider all such sets. Suppose we found such a set with the minimal total number of points. We show that there can be at most one distance that realizes the closest neighbors.

\begin{proposition}
    The closest distance is unique within the minimal set.
\end{proposition}

\begin{proof}
    Suppose we have several such distances. Pick the minimum one, which we denote by $d$. Pick the subset $S$ of points such that they have the closest neighbors at a distance $d$. Each point in $S$ will have 3 closest neighbors, all of which belong to $S$. Thus, we can remove all other points and get a smaller configuration $S$.
\end{proof}

We are now in a situation where the closest distance is the same for every point. We can assume this distance to be equal to one. We consider our points and connect the ones at distance $1$ with an edge, thus obtaining a graph.

Our graph is a penny graph. 
A \textit{penny graph} is a graph where two vertices are connected by an edge if and only if their distance is the minimum distance among all pairs of vertices. Such graphs are called penny graphs because we can place a penny at each vertex, so that the coins will not overlap and will touch each other if and only if the corresponding vertices are connected. It is assumed that the diameter of a penny is $1$. There is a large body of research on penny graphs (see, e.g., \cite{DE,AS,KS}, and references therein). 

It is clear that a penny graph is a \textit{unit distance graph}, which is defined as a graph formed from a collection of points in the Euclidean plane by connecting two points whenever the distance between them is exactly $1$. 

Moreover, a penny graph is a \textit{matchstick graph}: a graph that can be drawn in the plane in such a way that its edges are line segments of length~$1$ that do not cross each other. That is, it is a graph that has an embedding that is simultaneously a unit distance graph and a planar graph. Figure~\ref{fig:udm} shows two unit distance graphs: the one on the right is also a matchstick graph.

\begin{figure}[ht!]
    \centering
        \begin{tikzpicture}[scale=2, line join=round]
    \begin{scope}[rotate=0, shift={(0,0)}]
  \coordinate (A1) at ( 1/2, 0);
  \coordinate (B1) at ( 0, {sqrt(3)/2});
  \coordinate (C1) at (-1/2, 0);
  \draw[thick] (A1)--(B1)--(C1)--cycle;
  \draw[thick] (A1)--(C1);
  \foreach \P in {A1,B1,C1}
    \fill (\P) circle[radius=0.3mm];
    \end{scope}
    \begin{scope}[shift={({sqrt(3)/2},.5)}, rotate=0]
  \coordinate (A2) at ( 1/2, 0);
  \coordinate (B2) at ( 0, {sqrt(3)/2});
  \coordinate (C2) at (-1/2, 0);
  \draw[thick] (A2)--(B2)--(C2)--cycle;
  \draw[thick] (A2)--(C2);
  \foreach \P in {A2,B2,C2}
    \fill (\P) circle[radius=0.3mm];
    \end{scope}
    \draw[thick] (A2)--(A1);
    \draw[thick] (B2)--(B1);
    \draw[thick] (C2)--(C1);
    \end{tikzpicture}
    \qquad
    \begin{tikzpicture}[scale=2, line join=round]
    \begin{scope}[rotate=0]
  \coordinate (A1) at ( 1/2, 0);
  \coordinate (B1) at ( 0, {sqrt(3)/2});
  \coordinate (C1) at (-1/2, 0);
  \coordinate (D1) at ( 0, {-sqrt(3)/2});
  \draw[thick] (A1)--(B1)--(C1)--(D1)--cycle;
  \draw[thick] (A1)--(C1);
  \foreach \P in {A1,B1,C1,D1}
    \fill (\P) circle[radius=0.3mm];
    \end{scope}
    \begin{scope}[shift={({sqrt(3)/2},1/2)}, rotate=0]
  \coordinate (A2) at ( 1/2, 0);
  \coordinate (B2) at ( 0, {sqrt(3)/2});
  \coordinate (C2) at (-1/2, 0);
  \coordinate (D2) at ( 0, {-sqrt(3)/2});
  \draw[thick] (A2)--(B2)--(C2)--(D2)--cycle;
  \draw[thick] (A2)--(C2);
  \foreach \P in {A2,B2,C2,D2}
    \fill (\P) circle[radius=0.3mm];
    \end{scope}
    \draw[thick] (B2)--(B1);
    \draw[thick] (D2)--(D1);
    \end{tikzpicture}
    \caption{Two unit distance graphs with a matchstick graph on the right}
    \label{fig:udm}
\end{figure}
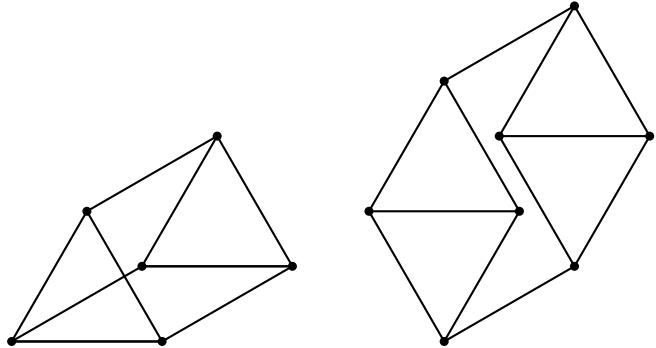

The graph on the left, the unit distance graph, has a planar embedding, but it does not have an isometric planar embedding, so it is not a matchstick graph.

Matchstick graphs allow two vertices to be at a smaller distance than the length of a matchstick. Thus, not every matchstick graph can be a penny graph, as, in a penny graph, the distance between any two vertices not connected by edges is longer than the length of an edge.

The minimal graph must be connected. Indeed, if it is disconnected, one component would keep the same properties and have fewer vertices. Thus, we will consider only connected penny graphs in the rest of the paper.

In our graph, each vertex has exactly $3$ neighbors. Such graphs are called \textit{$3$-regular graphs}.

By the degree sum formula, a $3$-regular graph with $n$ vertices has $\frac{3n}{2}$ edges. In particular, the number of vertices $n$ must be an even number. By the way, the minimal 3-regular matchstick graph has 8 vertices and is depicted on the right in Figure~\ref{fig:udm}.

To conclude, we want to prove that a 3-regular penny graph with 16 vertices, shown in Figure~\ref{fig:minimal}, is minimal.

\section{Building Blocks}

In our construction, we will use two unit distance graphs as building blocks. We call them a \textit{triangle} and a \textit{rhombus}. Figure~\ref{fig:trianglerhombus} shows both.

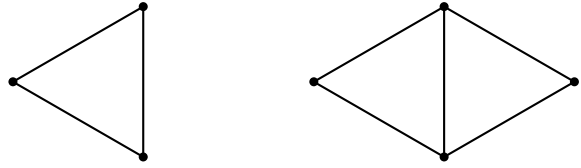
\begin{figure}[ht!]
    \centering
    \begin{tikzpicture}[scale=2, line join=round, rotate=90]
  \coordinate (E) at ( 0.5, 2);
  \coordinate (F) at ( 0, {sqrt(3)/2+2});
  \coordinate (G) at (-.5, 2);
  \draw[thick] (E)--(F)--(G)--cycle;
  \foreach \P in {E,F,G}
    \fill (\P) circle[radius=0.03];
  \coordinate (A) at ( 1/2, 0);
  \coordinate (B) at ( 0, {sqrt(3)/2});
  \coordinate (C) at (-1/2, 0);
  \coordinate (D) at ( 0, {-sqrt(3)/2});
  \draw[thick] (A)--(B)--(C)--(D)--cycle;
  \draw[thick] (A)--(C);
  \foreach \P in {A,B,C,D}
    \fill (\P) circle[radius=0.03];
\end{tikzpicture}
    \caption{Triangle and rhombus}
    \label{fig:trianglerhombus}
\end{figure}

Both constructions in Figures~\ref{fig:24pointssolution} and \ref{fig:minimal} consist of rhombuses connected to each other. The $24$-point construction consists of $6$ rhombuses, while the $16$-point construction consists of $4$ rhombuses.

What will happen if we try a similar construction with $3$ rhombuses? The image in Figure~\ref{fig:Threeboundaryrhombuses} shows such a configuration, which now has extra edges with the shortest distance. We now see $3$ points with more than three closest neighbors each, violating the condition. 

\begin{figure}[ht!]
    \centering
    \begin{tikzpicture}[scale=.8,  rotate=0]
   \coordinate (A) at ( 0, 0);
   \coordinate (B1) at ( 0, {1+sqrt(3)});
  \coordinate (B) at ( 0, 1);
  \coordinate (C) at ({sqrt(3)/2}, -1/2);
   \coordinate (C1) at ({3*sqrt(3)/2}, -3/2);
    \coordinate (D) at ({-sqrt(3)/2}, -1/2);
    \coordinate (D1) at ({-3*sqrt(3)/2}, -3/2);
  \draw[thick] (B)--(C)--(D)--cycle;
  \draw[line width=.3mm] (B)--(B1);
  \coordinate (K) at ( {sqrt(3)/2}, {-1/2-sqrt(3)});
  \coordinate (L) at ( {-sqrt(3)/2}, {-1/2-sqrt(3)});
  \foreach \P in {B,C,D, B1,C1,D1, K,L}
    \fill (\P) circle[radius=0.8mm];
  \draw[line width=.3mm] (K)--(C);
  \draw[line width=.3mm] (L)--(D);
  \draw[line width=.3mm] (K)--(L);
  \begin{scope}[rotate=120]
   \coordinate (C3) at ({sqrt(3)/2}, -1/2);
    \coordinate (D3) at ({-sqrt(3)/2}, -1/2);
  \coordinate (K3) at ( {sqrt(3)/2}, {-1/2-sqrt(3)});
  \coordinate (L3) at ( {-sqrt(3)/2}, {-1/2-sqrt(3)});
  \foreach \P in {K3,L3}
    \fill (\P) circle[radius=0.8mm];
  \draw[line width=.3mm] (K3)--(C3);
  \draw[line width=.3mm] (L3)--(D3);
  \draw[line width=.3mm] (K3)--(L3);
  \end{scope}
  \begin{scope}[rotate=-120]
   \coordinate (C2) at ({sqrt(3)/2}, -1/2);
    \coordinate (D2) at ({-sqrt(3)/2}, -1/2);
  \coordinate (K2) at ( {sqrt(3)/2}, {-1/2-sqrt(3)});
  \coordinate (L2) at ( {-sqrt(3)/2}, {-1/2-sqrt(3)});
  \foreach \P in {K2,L2}
    \fill (\P) circle[radius=0.8mm];
  \draw[line width=.3mm] (K2)--(C2);
  \draw[line width=.3mm] (L2)--(D2);
  \draw[line width=.3mm] (K2)--(L2);
  \end{scope}
  \draw[line width=.3mm] (L2)--(B1);
  \draw[line width=.3mm] (K3)--(B1);
  \draw[line width=.3mm] (L3)--(C1);
  \draw[line width=.3mm] (K)--(C1);
  \draw[line width=.3mm] (C)--(C1);
  \draw[line width=.3mm] (D)--(D1);
  \draw[line width=.3mm] (L)--(D1);
  \draw[line width=.3mm] (K2)--(D1);
\end{tikzpicture}
    \caption{A construction with $3$ rhombuses}
    \label{fig:Threeboundaryrhombuses}
\end{figure}
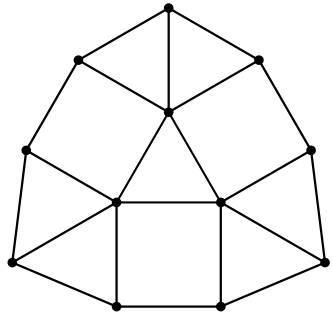

\section{Boundary}

An embedded planar graph divides the plane into regions called faces. Its \textit{boundary} is the set of its edges that are incident to the outer face. We call the edges on the boundary the \textit{outer edges}. All other edges will be referred to as \textit{inner edges}.

An edge is called a \textit{bridge} if the removal of it increases the number of connected components of the graph. A bridge that is an inner or an outer edge is called an \textit{inner} or an \textit{outer bridge}, respectively. The planar graph in Figure \ref{fig:bridges} has the inner bridges $AB$ and $KL$ and the outer bridges $CD$ and~$EF$.
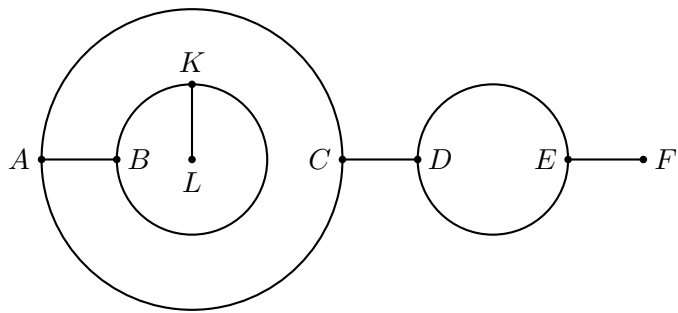
\begin{figure}[ht!]
\centering
    \begin{tikzpicture}[scale=1., line join=round]
    \coordinate (A) at ( -2, 0);
    \coordinate (B) at ( -1, 0);
    \coordinate (C) at (2,0);
    \coordinate (D) at (3, 0);
    \coordinate (E) at (5, 0);
    \coordinate (F) at (6, 0);
     \coordinate (K) at (0, 1);
      \coordinate (L) at (0, 0);
       \foreach \P in {A,B,C,D,E,F, K, L}
    \fill (\P) circle[radius=0.5mm];
     \node(A) at (-2.3, 0) {$A$};
      \node(B) at (-.7, 0) {$B$};
       \node(C) at (1.7, 0) {$C$};
        \node(D) at (3.3, 0) {$D$};
         \node(E) at (4.7, 0) {$E$};
          \node(F) at (6.3, 0) {$F$};
           \node(K) at (0, 1.3) {$K$};
            \node(L) at (0, -.3) {$L$};
        \draw [thick] (0, 0) circle (1);
        \draw [thick] (0, 0) circle (2);
        \draw [thick] (4, 0) circle (1);
        \draw [thick] (A)--(B);
    \draw [thick](C)--(D);
    \draw [thick](E)--(F);
     \draw [thick](K)--(L);
\end{tikzpicture}
\caption{Inner and outer bridges}
\label{fig:bridges}
\end{figure}

Consider a connected 3-regular penny graph. If all of its outer bridges are removed, we will be left with isolated vertices and subgraphs whose boundaries are polygons. We call these subgraphs {\it islands}.

These isolated points and islands connected by the outer bridges form a graph $T$. Each vertex of $T$ is an isolated point or an island. Each edge corresponds to an outer bridge of our original graph. That means that in $T$, every edge is a bridge, which implies that $T$ is a tree. A vertex of a tree that corresponds to an isolated vertex of the original graph has to have degree 3 in $T$. Hence, a leaf of the tree $T$ corresponds to an island with exactly one outer bridge connected to it. 

An example of such tree $T$ with one isolated vertex of degree $3$ shown as a solid dot and six islands represented by circles that are connected by bridges is given in Figure~\ref{fig:tree}. The leaf island circles are black, while the gray circles represent non-leaf islands.

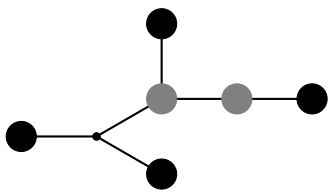
\begin{figure}[ht!]
\centering
    \begin{tikzpicture}[scale=1., line join=round]
    \coordinate (A) at ( 0, 0);
    \coordinate (B) at ( 1, 0);
    \coordinate (C0) at (1.866, 1.5);
    \coordinate (C1) at (1.866, .5);
       \coordinate (C2) at (1.866, -.5);
       \coordinate (D1) at (2.866, .5);
       \coordinate (E1) at (3.866, .5);
    \draw [thick] (A)--(B);
    \draw [thick](B)--(C1);
    \draw [thick](B)--(C2);
    \draw [thick](C0)--(C1);
    \draw [thick](C1)--(D1);
    \draw [thick](E1)--(D1);
    \filldraw[gray] (2.866, .5) circle (2mm);
     \filldraw (0,0) circle (2mm);
     \filldraw (1.866, 1.5) circle (2mm);
     \filldraw (3.866, .5) circle (2mm);
      \filldraw (1.866, -.5) circle (2mm);
  \filldraw (1,0) circle (.5mm);
  \filldraw[gray] (1.866,.5) circle (2mm);
\end{tikzpicture}
\caption{A tree $T$}
\label{fig:tree}
\end{figure}

An example of a leaf island of a 3-regular penny graph with an outer bridge shown by a dashed line is given in Figure \ref{fig:leaf}.

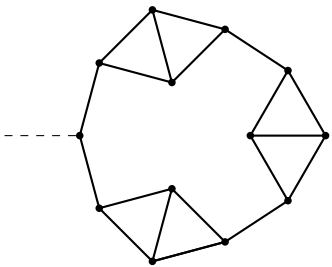
\begin{figure}[ht!]
\centering
    \begin{tikzpicture}[scale=1., line join=round]
  \coordinate (O) at ( -1, 0);
  \coordinate (A) at ( 0, 0);
  \coordinate (B1) at ( .259, .966);
   \coordinate (B2) at ( .259, -.966);
  \coordinate (C1) at (1.225, .707);
  \coordinate (C2) at (1.225, -.707);
  \coordinate (D1) at (.966, 1.673);
  \coordinate (D2) at (.966, -1.673);
  \coordinate (E1) at (1.932, 1.414);
   \coordinate (E2) at (1.932, -1.414);
     \coordinate (F1) at (2.768, .866);
     \coordinate (F2) at (2.768, -.866);
      \coordinate (G1) at (2.268, 0);
      \coordinate (G2) at (3.268, 0);
    \foreach \P in {A,B1,B2,C1,C2,D1,D2,E1,E2,F1,F2,G1,G2}
    \fill (\P) circle[radius=0.5mm];
\draw[dashed] (O)--(A);
    \draw [thick] (A)--(B1);
    \draw [thick](A)--(B2);
    \draw [thick](C1)--(B1);
    \draw [thick](C2)--(B2);
    \draw [thick](D1)--(B1);
    \draw [thick](D1)--(C1);
    \draw [thick](D2)--(B2);
    \draw [thick](D2)--(C2);
    \draw [thick](E2)--(C2);
    \draw [thick](D2)--(E2);
    \draw [thick] (E1)--(C1);
    \draw [thick](D1)--(E1);
    \draw [thick] (F1)--(E1);
    \draw [thick] (D2)--(E2);
    \draw [thick](F2)--(E2);
    \draw [thick](F2)--(G2);
    \draw [thick](F2)--(G1);
    \draw [thick](G1)--(G2);
    \draw [thick](F1)--(G2);
    \draw [thick](F1)--(G1);
\end{tikzpicture}
\caption{A leaf island}
\label{fig:leaf}
\end{figure}

First, we will prove that a 3-regular penny graph with no outer bridges has at least 16 vertices.

\section{3-regular Penny Graphs with no Outer Bridges}

Consider a 3-regular penny graph with no outer bridges. Its boundary is an $n$-gon whose sides are edges of length 1. In Figure~\ref{fig:boundarytypes}, we give an example of a portion of such a boundary made of thick edges, where the edges that are not incident to the boundary are not shown. Each outer edge is adjacent to two inner edges. We say that the sum of the angles between an outer edge and the adjacent inner edges is the {\it angle contribution} of the outer edge. The sum of angle contributions of all the outer edges is equal to the sum of the interior angles of the boundary $n$-gon, which is $180^\circ (n-2)$.

\begin{figure}[ht!]
\centering
    \begin{tikzpicture}[scale=1.3, rotate=-90]
    \coordinate (K) at ( 0, {-1-sqrt(3)/2});
     \coordinate (L) at ( 1, {-1-sqrt(3)/2});
     \coordinate (M) at ( 1/2, {-1-sqrt(3)});
     \coordinate (N) at ( 1/2, {-2-sqrt(3)});
    \coordinate (A) at ( 1/2, 0);
  \coordinate (B) at ( 0, {sqrt(3)/2});
  \coordinate (C) at (-1/2, 0);
  \coordinate (D) at ( 0, {-sqrt(3)/2});
  \coordinate (E) at ( {sqrt(2)/2}, {sqrt(2)/2 + sqrt(3)/2});
  \coordinate (F) at ( {1/2+sqrt(2)/2}, {sqrt(3)+sqrt(2)/2});
   \coordinate (G) at ( {sqrt(2)}, {sqrt(3)/2});
    \foreach \P in {A,B,C,D,E,F,G,K,L,M,N}
    \fill (\P) circle[radius=0.5mm];
    \draw[line width=.5mm] (K)--(D);
  \draw[line width=.5mm] (B)--(C);
  \draw[line width=.5mm] (C)--(D);
  \draw[line width=.1mm] (A)--(B);
   \draw[line width=.1mm] (D)--(A);
    \draw[line width=.1mm] (A)--(C);
     \draw[line width=.5mm] (B)--(E);
       \draw[line width=.5mm] (E)--(F);
        \draw[line width=.1mm] (E)--(G);
        \draw[line width=.1mm] (K)--(L);
        \draw[line width=.5mm] (K)--(M);
         \draw[line width=.1mm] (L)--(M);
           \draw[line width=.5mm] (N)--(M);
    \end{tikzpicture}
    \caption{Different types of outer edges}
    \label{fig:boundarytypes}
\end{figure}
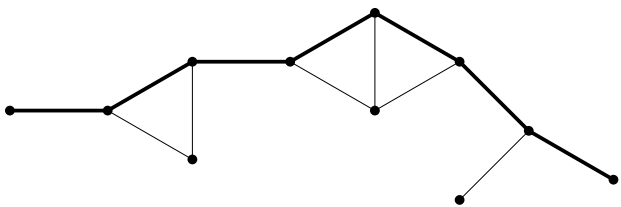

Suppose that there are $k$ sides of this boundary polygon that are edges of equilateral triangle faces of the graph. We call such outer edges \textit{triangle edges}, and we call the other $n-k$ outer edges \textit{free edges}. Each triangle edge contributes $120^\circ$ to the sum of interior angles. Now we want to look at the angle contribution of the free edges.

To this end, we will use the following geometric lemma.

\begin{lemma}
\label{lem:geom}
Let $KABL$ be a quadrilateral such that $|KA|=|AB|=|BL|=1$, $\angle KAB >60^\circ$, $\angle ABL >60^\circ$, and $\angle KAB + \angle ABL < 180^\circ$. Then $|KL|< 1$.
\end{lemma}

We asked ChatGPT to prove this lemma and verified the proof. The proof is technical and not very elegant, so we omit it. Instead, we show an image illustrating Lemma~\ref{lem:geom} in Figure~\ref{fig:proofoflemma}, so the statement is easier to follow.

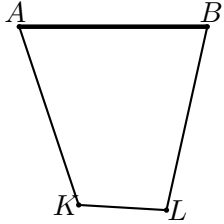
\begin{figure}[ht!]
\centering
    \begin{tikzpicture}[scale=0.5, line join=round, rotate=90]
  \coordinate (K) at ( 4.1191, -.9153);
  \coordinate (A) at ( 9, -2.);
  \coordinate (B) at ( 9, 3);
  \coordinate (L) at (4.2566, 1.4189);
  \foreach \P in {K, B, A, L}
    \fill (\P) circle[radius=0.7mm];
    \draw [thick](L)--(K)--(A)--(B)--cycle;
    \draw[line width=.6mm] (B)--(A);
    \node(L) at ( 4.1191, -1.2) {$L$};
    \node(B) at (9.4,-2.1) {$B$};
    \node(A) at (9.4, 3.1) {$A$};
    \node(K) at (4.2566, 1.75) {$K$};
\end{tikzpicture}
\caption{The quadrilateral in Lemma~\ref{lem:geom}}
\label{fig:proofoflemma}
\end{figure}

\begin{lemma}
\label{lem:freeedge}
A free edge contributes at least $180^\circ$ to the sum of the interior angles.
\end{lemma}

\begin{proof}
Let $AB$ be a free edge. Each vertex on the boundary has exactly one inner edge that comes out of it. Let $AK$ and $BL$ be such inner edges. The fact that $AB$ is a free edge guarantees that the vertices $K$ and $L$ are distinct. Therefore, $|AL|>1$ and $|BK|>1$, which implies that $\angle ABL > 60^\circ$ and $\angle KAB>60^\circ$. We also know that $KABL$ is a quadrilateral with $|KA|=|AB|=|BL|=1$. The free edge $AB$ contributes the angle $\angle KAB + \angle ABL$ to the sum of the interior angles. Suppose that this contribution is less than $180^\circ$. Then the quadrilateral $KABL$ satisfies the assumptions of Lemma~\ref{lem:geom}, which implies that $|KL| < 1$. This contradicts the fact that no two vertices are at distances shorter than~$1$.
\end{proof}

We are ready to prove our conjecture for a graph with no outer bridges.

\begin{theorem}\label{thm:main}
    A minimal 3-regular penny graph with no outer bridges has $16$ vertices.
\end{theorem}

\begin{proof}
Consider the boundary of such a graph. Each triangle edge contributes $120^\circ$ to the sum of interior angles of the boundary polygon. Lemma~\ref{lem:freeedge} shows that a free edge contributes at least~$180^\circ$.

Recall that our boundary polygon has $n$ edges with $k$ triangle edges. By subtracting the angle contribution of the triangle edges, we get that the total angle contribution of the $n-k$ free edges is equal to $180(n-2) - 120k$ degrees. On the other hand, it is at least $180(n-k)$ degrees. Therefore, we have the inequality
\begin{equation}\label{eqn:ineq}
    180(n-2) - 120k \geq 180(n-k),
\end{equation}
which simplifies to
\[
k \geq 6.
\]

We want to prove by contradiction that a 3-regular penny graph with no outer bridges has at least 16 vertices. Since the number of vertices of a 3-regular graph must be even, we will assume that our penny graph has at most 14 vertices.

Since each boundary rhombus has $2$ outer edges and $4$ vertices, we can say that each of its outer edges contributes $2$ vertices to the graph. Each triangle edge that is not part of a rhombus contributes $3$ vertices to the graph. Thus, each triangle edge contributes at least 2 vertices, which implies that $2k \le 14$.

Therefore, $6 \le k \le 7$. Suppose $k=7$, then each triangle edge contributes exactly two vertices, and therefore belongs to a rhombus. Since each boundary rhombus has exactly $2$ triangle edges, the total number of triangle edges that belong to rhombuses is even, and cannot be equal to 7. 

Thus, $k=6$, and Eq.~\ref{eqn:ineq} is tight, implying that each free edge contributes exactly $180$ degrees.

Suppose the number of triangle edges that belong to rhombuses is $r$, then the number of triangle edges that do not belong to rhombuses is $6-r$, and all of them contribute 
\[3(6-r) + 2r = 18 - r\]
vertices. As this number is not more than 14, we get that $r \geq 4$. Therefore, there are at least two boundary rhombuses.

Suppose that a free edge $AB$ with inner edges $AK$ and $BL$ is adjacent to a boundary rhombus. Suppose that the vertex $B$ belongs to a rhombus, implying that $L$ is also part of that rhombus and has $3$ edges inside the rhombus. If $\angle KAB + \angle ABL = 180^\circ$, then $KABL$ is a parallelogram and $|KL| = 1$, implying that $L$ has an extra edge, leading to a contradiction, as shown in Figure~\ref{fig:proofoftheorem}.

\begin{figure}[ht!]
\centering
    \begin{tikzpicture}[scale=0.4]
 \coordinate (K) at (-.8682, -4.924);
  \coordinate (L) at ( 4.132, -4.924);
  \coordinate (A) at ( 0, 0);
  \coordinate (B) at ( 5, 0);
  \coordinate (C) at (8.83, -3.214);
  \coordinate (D) at (7.962, -8.138);
  \foreach \P in {K, L, B, C, A, D}
    \fill (\P) circle[radius=1.5mm];
    \draw [line width=.6mm](B)--(A);
    \draw [line width=.6mm](C)--(B);
    \draw[thick] (A)--(K);
    \draw[thick] (L)--(B);
    \draw[thick] (C)--(L);
    \draw[thick] (L)--(K);
    \draw[thick] (L)--(D);
    \draw[line width=.6mm] (C)--(D);
    \node(K) at ( -.95, -5.6) {$K$};
    \node(L) at (4, -5.6) {$L$};
    \node(A) at (0,.7) {$A$};
    \node(B) at (5.1, .7) {$B$};
    \end{tikzpicture}
\caption{The boundary in Theorem \ref{thm:main}}
\label{fig:proofoftheorem}
\end{figure}
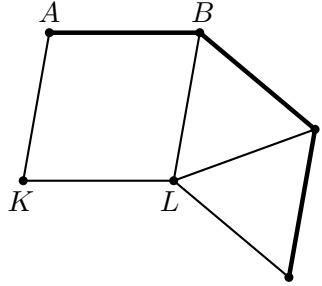

Figure~\ref{fig:minimal} shows a $3$-regular penny graph with $16$ vertices and no outer bridges. Thus, $16$ is the smallest number of vertices in a $3$-regular penny graph with no outer bridges.
\end{proof}

\section{Penny Graphs with Outer Bridges}

If a 3-regular penny graph has an outer bridge, then it has at least two leaf islands. The vertex incident to the outer bridge in each leaf island is of degree $2$ inside the island, and all other vertices are of degree~$3$. As the sum of the degrees of all vertices needs to be even, it follows that the number of vertices of degree $3$ must be even, and therefore, the total number of vertices in a leaf island is odd.

In order to prove that a 3-regular penny graph with outer bridges has at least $16$ vertices, it suffices to prove that a leaf island has more than $7$ vertices.

\begin{lemma}\label{lemma:leafisland}
A leaf island has more than $7$ vertices.
\end{lemma}
\begin{proof}
Suppose a leaf island has at most $7$ vertices and $n$ boundary vertices.

Each boundary vertex of degree $3$ contributes at least $120^\circ$ to the sum of the interior angles of the boundary polygon, while the boundary vertex of degree $2$ contributes at least $60^\circ$. Thus, we obtain an inequality
\[
180(n-2) \geq 120(n-1) + 60,
\]
which simplifies to $n \geq 5$. If $n=5$, then the interior angles of the boundary pentagon at the vertices of degree $3$ are equal to $120^\circ$, and the angle at the vertex of degree 2 is equal to $60^\circ$. A pentagon with such angles cannot be equilateral. If the three sides adjacent to $120^\circ$ angles on both ends are of length $1$, then the two sides adjacent to the $60^\circ$ angle are of length $2$, as shown in Figure~\ref{fig:pentagon}. Therefore, $n \geq 6$.
\begin{figure}[ht!]
\centering
    \begin{tikzpicture}[scale=1]
     \coordinate (A) at (0,0);
      \coordinate (B) at ( 1, 0);
       \coordinate (C) at (2, 0);
     \coordinate (D) at ( .5, -.866);
     \coordinate (E) at ( 1.5, -.866);
     \coordinate (F) at ( 1, 1.732);
      \coordinate (G) at ( 1, 2.732);
 \foreach \P in {C, A, D, E, F}
    \fill (\P) circle[radius=0.7mm];
\draw [line width=.4mm](D)--(A);
\draw [line width=.4mm](E)--(D);
\draw [line width=.4mm](E)--(C);
\draw [line width=.4mm](F)--(A);
\draw [line width=.4mm](C)--(F);
    \end{tikzpicture}
\caption{A pentagon from the proof of Lemma \ref{lemma:leafisland}}
\label{fig:pentagon}
\end{figure}
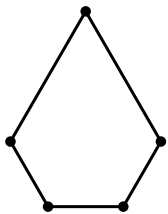

As the total number of vertices in a leaf island is odd, we have two cases:
\begin{itemize}
    \item $n=7$, and the island does not have any other vertices.
    \item $n=6$, and there is one inner vertex.
\end{itemize}

Suppose $n=7$, then all inner edges need to connect boundary vertices to each other. Let us label the boundary vertices in order as $A_1$, $A_2$, $A_3$, $A_4$, $A_5$, $A_6$, and $A_7$, with vertex $A_1$ of degree $2$. Consider the edge coming out of $A_4$. It cannot connect to $A_2$, as, in this case, the inner edge out of $A_3$ is blocked. It cannot connect to $A_6$ for a similar reason. It cannot connect to $A_7$, as the inner edges coming out of $A_2$ and $A_3$ are blocked.

It follows that $n=6$. We label the boundary vertices in order as $A_1$, $A_2$, $A_3$, $A_4$, $A_5$, and $A_6$, with vertex $A_1$ of degree $2$. There is also one inner vertex labeled by $B$, which is connected to $3$ boundary vertices, and the other two boundary vertices are connected to each other. By a similar argument to the case of $n=7$, the inner vertex $B$ needs to be connected to $A_3$, $A_4$, and $A_5$, while $A_2$ and $A_6$ have to be connected to each other. Therefore, the vertices $A_1$, $A_2$, and $A_3$ form an equilateral triangle, and the remaining four vertices form a rhombus. 

Consider the equilateral pentagon $A_2A_3BA_5A_6$. We have $\angle A_3BA_5 = 240^\circ$. As the distance $|BA_2| > 1$, implying that $\angle A_2A_3B > 60^\circ$. Similarly, $\angle BA_5A_6 > 60^\circ$. In addition, $\angle A_6A_2A_3 > 60^\circ$ and $\angle A_5A_6A_2 > 60^\circ$, and $|A_3A_5| > 1$, which allows us to apply Lemma~\ref{lem:geom} to conclude that $\angle A_6A_2A_3 + \angle A_5A_6A_2 > 180^\circ$. Summing up all the interior angles in the pentagon $A_2A_3BA_5A_6$, we get that their total sum is greater than $240 + 60 + 60 + 180 = 540$ degrees. Thus, such a pentagon does not exist.

We conclude that the number of vertices of a leaf island is greater than~$7$.
\end{proof}

Combining the results of Theorem~\ref{thm:main} and of Lemma~\ref{lemma:leafisland}, we obtain the following theorem.

\begin{theorem}
    A minimal 3-regular penny graph has 16 vertices.
\end{theorem}

\section{Acknowledgments}
When we started writing this paper, we did not know the names of unit distance, matchstick, and penny graphs. We also did not know whether our theorem was new. We asked ChatGPT, but it was not helpful. After we figured out the names, we found several experts in penny graphs and asked whether our theorem was known; they said no. The authors thank David Eppstein, Arsenii Sagdeev, and Konrad Swanepoel for their prompt replies. Konrad Swanepoel also notified us that in the first version of our preprint, we considered only penny graphs without outer bridges, and we are very grateful to him for that.

\noindent
Tanya Khovanova \\
\textsc{
Department of Mathematics, Massachusetts Institute of Technology\\
77 Massachusetts Avenue, Building 2, Cambridge, MA, U.S.A. 02139}\\
\textit{E-mail address: }\texttt{tanyakh@yahoo.com}

\medskip

\noindent Alexander Karabegov \\
\textsc{
Department of Mathematics, Abilene Christian University\\
1695 Campus Ct, Onstead Science Center,
Abilene, TX, U.S.A. 79699}\\
\textit{E-mail address: }\texttt{axk02d@acu.edu}

\begin{thebibliography}{9}

\bibitem{DE} Eppstein, David (2018), Edge Bounds and Degeneracy of Triangle-Free Penny Graphs and Squaregraphs, {\it Journal of Graph Algorithms and Applications}, 22(3), 483--499. 

\bibitem{GT} Galperin, Grigorii and Tolpygo, Aleksei (1986), Moscow mathematical olympiads, Moscow, Prosveshchenie publ., 303 p. (in Russian).

\bibitem{AS} Sagdeev, Arsenii (2025), General penny graphs are at most $43/18$-dense, {\it Combinatorics and Number Theory}, Vol.~14, No.~1, 75--89.
\bibitem{KS} Swanepoel, Konrad J.\ (2009), Triangle-free minimum distance graphs in the plane, {\it Geombinatorics}, 19(1): 28--30.
\end{thebibliography}
\end{document}